\documentclass[11pt,a4paper]{article}
\usepackage{array,epsfig,rotating,enumerate,dsfont,caption}
\usepackage{hyperref}
\hypersetup{colorlinks=true,citecolor= blue,linktocpage=true,linkcolor=blue}

\usepackage{bbm}
\usepackage{tikz}
\usepackage{xcolor}
\usepackage{mathtools}
\usepackage{marvosym}
\usepackage{amsmath,color}
\usepackage{amssymb,latexsym}
\usepackage{amsfonts}
\usepackage{amsthm}
\usepackage{dirtytalk}
\usepackage{graphicx}
\usepackage{caption}
\usepackage{comment}
\usepackage{subcaption}
\usepackage{pdflscape}
\usepackage{cancel}
\newtheorem{theorem}{Theorem}

\newtheorem{lemma}[theorem]{Lemma}
\newtheorem{definition}[theorem]{Definition}

\newtheorem{proposition}[theorem]{Proposition}
\theoremstyle{definition}

\usepackage{setspace}
\usepackage[mathscr]{euscript}
\usepackage[margin=1in]{geometry}

\title{Critical threshold 
for regular graphs}  

\author{Ishaan Bhadoo}

\date{}

\begin{document}

\maketitle

\begin{abstract}
In this article, we study the critical percolation threshold $p_c$ for $d$-regular graphs. It is well-known that $p_c \geq \frac{1}{d-1}$ for such graphs, with equality holding for the $d$-regular tree. We prove that among all quasi-transitive $d$-regular graphs, the equality $p_c(G) = \frac{1}{d-1}$ holds if and only if $G$ is a tree. Furthermore, we provide counterexamples that illustrate the necessity of the quasi-transitive assumption.

    \end{abstract}
    
\section{Introduction}
Consider independent (Bernoulli) bond\footnote{Although we work with bond percolation, the same methods apply to site percolation.} percolation on a locally finite, connected, infinite simple graph $G$ (all graphs are assumed to satisfy these conditions unless stated otherwise), i.e. we retain edges with probability $p$ and throw them away with probability $1-p$.

\vspace{2mm}
We use $\mathbb{P}_p$ to denote the corresponding percolation measure. An important function to consider in the context of percolation is, $\psi(p) = \mathbb{P}_p(\exists \ \text{an infinite connected component})$. This leads to the definition of the critical parameter: $$p_c:= \sup \{p: \psi(p) = 0\}$$ 

Determining the exact value of $p_c$ is challenging for most graphs; however for $d$-regular trees, one can show that $p_c = \frac {1}{d-1}$. The proof follows two steps: First, by a first-moment argument one can show that a graph $G$ with maximal degree $d < \infty$ satisfies $p_c(G) \geq \frac{1}{d - 1}.$ Second, for trees with degree $d$, by a dual second moment upper bound we can get $p_c \leq \frac{1}{d -1}$,  implying $p_c = \frac{1}{d -1}$ (see \cite[\textcolor{blue}{Claim 2.3.9}]{roch2024modern} for details). 

\vspace{2mm}
This leads to the following question: consider percolation on a $d$-regular graph $G$, are trees the only graphs with $p_c = \frac{1}{d-1}?$

\vspace{2mm}

The goal of this article is to show that trees are the only graphs with $p_c = \frac{1}{d-1}$ in the space of all $d$-regular quasi-transitive graphs. We start with some definitions. For a graph $G$, let $\mathrm{AUT}(G)$ be the group of all automorphisms (adjacency preserving bijections) of $G.$

\begin{definition}
    A graph $G$ is called \textit{quasi-transitive} if the number of orbits for the action of $\mathrm{AUT}(G)$ on $G$ is finite. It is called \textit{transitive} if there is only one orbit.
\end{definition}
Under the assumption of quasi-transitivity we have the following theorem:

\begin{theorem}
    Let $G$ be a quasi-transitive $d$-regular graph. Then $p_c(G) \geq \frac{1}{d - 1}$ and equality holds if and only if $G$ is a tree.
\end{theorem}

It is important to note that being quasi-transitive is essential. In Section 2, using the general theory for percolation on trees we give a counterexample (for each $d \geq 3$) when one drops this assumption. Next, we show the above theorem by constructing a covering of every quasi-transitive $d$-regular graph using regular trees and then use the strict monotonicity result of \cite{martineau2019strict}. The main tools we use are from \cite{martineau2019strict} and \cite{lyons2017probability}. For completeness, we cover the required background for percolation on trees and some techniques from the theory of coverings, in an effort to make this article self-contained.

\subsection{Connection to the connective constant}
A self-avoiding walk is a path that visits no vertex more than once. To define the connective constant, fix a starting vertex $o$, the set of all self-avoiding walks of length $n$ starting at $o$ is denoted as $\text{SAW}_n$. The connective constant $\mu(G)$ of a graph $G$ is then defined as $$\mu(G) := \underset{n\rightarrow\infty}{\lim} |\text{SAW}_n|^{\frac{1}{n}}$$

By Fekete's lemma, it can be checked that this limit exists. The connective constant is closely related to the critical threshold by the following standard lemma (see \cite{lyons2017probability}).

\begin{lemma}
    For any connected infinite graph $G,\ p_c(G) \geq \frac{1}{\mu(G)}\footnote{$\text{This also shows} \ p_c(G) \geq \frac{1}{d-1} \text{for a graph with maximum degree} \ d.$}.$ 
\end{lemma}

\begin{proof}
 Let \(\mathcal{C}(o)\) denote the connected component of \(o\) in Bernoulli percolation with parameter \(p\). Define \(S_n(o)\) as the set of self-avoiding walks of length \(n\) within \(\mathcal{C}(o)\). If \(\mathcal{C}(o)\) is infinite, then \(S_n(o) \neq \emptyset\) for all \(n\). From this, we deduce:
\[
\mathbb{P}_p(o\leftrightarrow\infty) \leq \mathbb{P}_p[S_n(o) \neq \emptyset] \overset{\text{(Markov's Ineq)}}{\leq} \mathbb{E}_p[|S_n(o)|] = |\text{SAW}_n| p^n.
\]

By taking $n$-th roots we get, $1 \leq \mu(G)p \ \text{whenever } \mathbb{P}_p(o\leftrightarrow\infty) > 0.$
In particular, this holds for $p > p_c$.
\end{proof}

An analogous statement to Theorem 2 regarding the connective constant was previously shown by Grimmett and Li \cite[\textcolor{blue}{Thm 4.2}]{grimmett2015bounds}. 
\begin{theorem}[G. Grimmett, Z. Li, \cite{grimmett2015bounds}]
    Let \( G = (V, E) \) be a $d$-regular quasi-transitive graph  and let \(d \geq 3\). We have that \(\mu(G) < d- 1\) if $G$ has cycles.
\end{theorem}
By using Lemma 3, Theorem 2 follows. However, the techniques used in the proof of Theorem 4 are entirely combinatorial and therefore differ from our method.

\vspace{2mm}
\section{Percolation on trees}
In this section for $d \geq 3$ we give an example of a $d$-regular graph with cycles such that $p_c = \frac{1}{d -1}$ (such an example cannot exist for $d = 2)$.  To do this we use the theory of percolation on locally finite trees. We start by defining the branching number of a locally finite tree. 

\vspace{2mm}
\subsection{Branching number and the critical point for trees}
Suppose $T = (V, E)$ is an infinite locally finite tree with root $O$. We imagine the tree $T$ as growing downward from the root $O$. For $x, y \in V$, we write $ x \leq y$ if $x$ is on the shortest path from $O$ to $y$; and $T_x$ for the subtree of $T$ containing all the vertices $y$ with $y\geq x$. For a vertex $x \in V$ we denote by $d(x,O)$ the graph distance from $O$ to $x$. We want to understand the critical point for a tree, motivated by the comparison from Galton-Watson branching processes this naturally leads us to the study of the average number of branches coming out of a vertex which is called the  branching number of a tree. To rigorously define this we use conductances and flows on trees. For each edge $e$ we define the conductance of an edge to be $c(e):= \lambda^{-|e|}$, where $|e|$ denotes the distance of the edge $e$ from the root $O.$ It is natural to define conductances decreasing exponentially with the distance since trees grow exponentially. 

\vspace{2mm}
If $\lambda$ is very small then due to large conductances there is a non-zero flow on the tree satisfying $0 \leq \theta(e) \leq \lambda^{-|e|} $. While increasing the value of $\lambda$ we observe a critical value $\lambda_c$ above which such a flow does not exists. This is precisely the branching number. Specifically, $$br(T) := \sup \{\lambda : \exists \ \text{a non-zero flow}\ \theta \ \text{on} \ T \ \text{such that} \ 0 \leq \theta(e) \leq \lambda^{-|e|} \  \forall \ e \in T \}$$

By using the max-flow min-cut theorem we get that, 

$$br(T) = \sup \{\lambda : \inf_{\pi} \displaystyle \sum_{e \in \pi} \lambda^{-|e|} > 0 \}$$

Where the infimum is over all cutsets $\pi$ separating $O$ from $\infty.$ Using this as the definition it is easy to see that $p_c \geq \frac{1}{br(T)}$, indeed by using a first moment bound at $\lambda = \frac{1}{p}$ for $p > p_c.$ By using a (weighted) second-moment method it can be shown that the reverse inequality also holds. In particular, we have the following result of Lyons. 
\begin{theorem} [R. Lyons, \cite{lyons1990random}]
\label{brt}
    Let $T$ be a locally finite, infinite tree then, $p_c(T) = \frac{1}{br(T)}$ where $br(T)$ is the branching number of the tree.
\end{theorem} 
\emph{Proof:} The proof essentially uses a lower bound on being connected to infinity in terms of conductances \cite{lyons1992random}. See \cite{lyons2017probability} for the proof.

\vspace{2mm}
Thus, to find the critical threshold for a tree one needs to know how to compute its branching number. However, the definition of the branching number makes this in general hard, thankfully, for sub-periodic trees (defined below) we have a significantly easier method of calculating the branching number.

\subsection{Superiodic trees}
For a tree $T$ we define its upper exponential growth rate as $$\overline{grT} := \limsup_{n \rightarrow \infty} |T_n|^{\frac{1}{n}}$$ where $T_n$ is the number of vertices at a distance $n$ from $O.$ Similarly one can define the lower exponential growth rate as $$\underline{grT} := \liminf_{n \rightarrow \infty} |T_n|^{\frac{1}{n}}$$ We say that the exponential growth rate exists if $ \overline{grT}= \underline{grT}$. 

\vspace{2mm}
We now recall the definition of subperiodic trees from \cite{lyons2017probability}. Fix a $N\geq 0$. An infinite tree $T$ is called \textbf{$N$- subperiodic} if $\forall x \in T$ there exists an adjacency preserving injection $f: T_x \rightarrow T_{f(x)}$ with $|f(x)| \leq N$ (where $|\cdot|$ is the distance from $O$). A tree is called \textbf{subperiodic} if there exists a $N$ for which it is $N$-subperiodic. Since in general the growth rate is easier to calculate, the following theorem is the key to calculating $p_c$ for subperiodic trees.

\begin{theorem}[Subperiodicity and Branching Number, \cite{lyons2017probability}]
    For every subperiodic infinite tree $T$, the exponential growth rate exists and $\text{br}T = \text{gr}T.$
\end{theorem}
\subsection{Non quasi-transitive counter-examples}

\begin{figure}
\centering    
 \includegraphics[scale = 0.40]{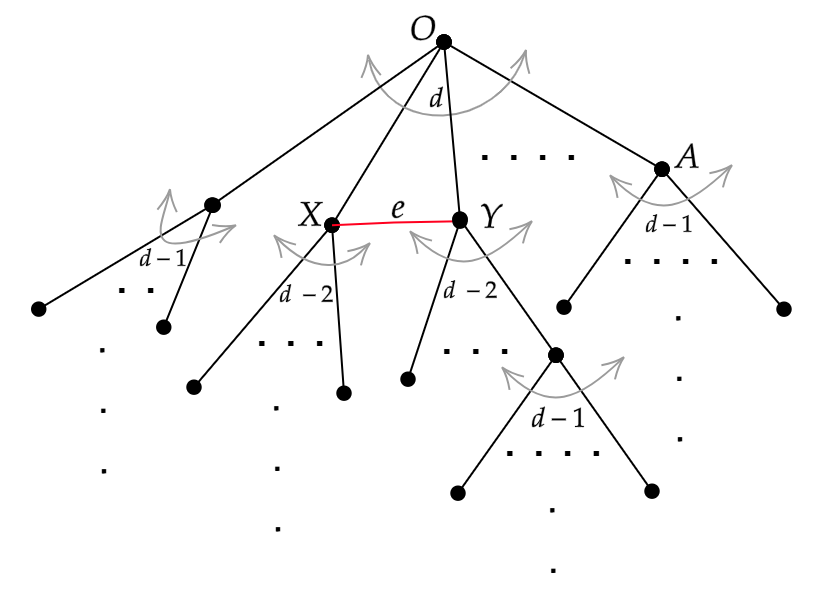}
 \caption{On removing the edge $e$ we get a sub-periodic tree $T$ with gr$T=$ br$T=d -1$}
    \label{fig:enter-label}

\end{figure}

We are now ready to give our counterexamples. Let $T$ be a tree with root $O$ such that every vertex in $T$ has degree $d$, except two vertices $X$, $Y$ that are adjacent to the root having degree $d -1$. Hence, $T = (V,E)$ is the graph formed by all black edges shown in Figure 1. Now define $G:= T \cup \{e\} = (V,E \cup \{e\})$ to be the graph obtained after adding the red edge e. We claim that $p_c(G) = \frac{1}{d-1}.$
 
\vspace{2mm}
 For the tree $T$, $|T_1| = d$, $|T_2| = (d-2)(d -1) + 2(d -2) = (d -2)(d +1)$, after this point every point has $d -1$ branches coming out, so $|T_{2+n}| = (d -2)(d +1)(d-1)^n$. Therefore $\text{gr}T = d -1.$ 

\vspace{2mm}
$T$ is clearly subperiodic since, for all $x$ such that $d_T(x, O) \geq 2$, we have $T_x$ is exactly $T_A$, allowing us to define the function $f(v) = \phi(v)$, where $\phi$ is the isomorphism between $T_x$ and $T_A$. Thus, $T$ is 1-subperiodic. By Theorem 6, $\text{br}T = \text{gr}T = d - 1$. Hence, $p_c(T) = \frac{1}{d - 1}.$

\vspace{2mm}
Since $T$ is a subgraph of $G$, we have $p_c(G) \leq p_c(T) = \frac{1}{d-1}$. However, $G$ is of degree $d$, so by a standard first-moment bound, $p_c(G) \geq \frac{1}{d - 1}$. Combining these two observations, we conclude that $p_c(G) = \frac{1}{d - 1}$. Thus, $G$ is a $d$-regular graph with cycles such that $p_c(G) = \frac{1}{d - 1}$. Finally, the fact that $G$ is not quasi-transitive follows from Theorem 2.

\vspace{2mm}

\section{Proof of the Theorem}
We now show that if $G$ is a quasi-transitive, $d$-regular graph then $p_c(G) > \frac{1}{d -1}.$ The key idea is to cover every quasi-transitive $d$-regular graph (with cycles) by a $d$-regular tree. We start by defining what a covering map means in the context of percolation, next we use the results of Martineau and Severo \cite{martineau2019strict} about critical thresholds under coverings.

\subsection{Critical points under coverings}
The question of critical points under coverings was asked by Benjamini and Schramm in their celebrated paper ``Percolation beyond $\mathbb{Z}^d,$ Many Questions and a Few Answers" \cite[\textcolor{blue}{Question~1}]{benjamini1996percolation}. 
They conjectured that if $G, H$ are quasi-transitive graphs and $G$ covers but is not isomorphic to $H$ and $p_c(G) < 1 $ then $p_c(G) < p_c(H)$. This conjecture was resolved by Martineau and Severo \cite{martineau2019strict}. Following their paper we set up some definitions necessary to define a covering map.

\vspace{2mm}
 Consider a map $\pi:V(G) \rightarrow V(H)$, we say that this map is a \textit{strong covering map} if its $1$-Lipschitz (i.e. $d_H(\pi(x), \pi(y)) \leq d_G(x,y)$) and it has the \textit{strong lifting property}: for every $x \in V(G)$, and for every neighbour $u$ of $\pi(x)$ there is a unique neighbour of $x$ that maps to $u$. Next we say that a map $\pi : V(G) \rightarrow V(H)$  has \textit{uniformly non-trivial fibres} (\cite{severo2020interpolation}) if there exists $R$ such that for all $x \in V(G)$ there exists $y \in V(G)$ such that $\pi(x) = \pi(y)$ and $0< d_G(x,y) \leq R$. We are now ready to state the main tool: 

\begin{theorem} [F.Severo, S. Martineau, \cite{martineau2019strict}]
    Let $G$ and $H$ be graphs of bounded such that there is a map $\pi: V(G) \rightarrow V(H)$ which is a strong covering map with uniformly non-trivial fibres. Then if $p_c(G) < 1$, we have $p_c(G) < p_c(H).$ 
 \end{theorem}

The above result relies on the theory of enhancements. A technique first introduced by Aizenman and Grimmett \cite{aizenman1991strict} as a recipe to prove strict inequalities between critical points of graphs, and is part of a more general idea of interpolation between percolation configurations \cite{severo2020interpolation}. For background on the technique of enhancements see \cite{severo2020interpolation}, \cite{balister2014essential}. We now show that this holds for $G = d$-regular tree and $H$ a $d$-regular quasi-transitive graph with cycles. In particular, we have the following:

 \begin{proposition}
         Let $T_{d}$ be the $d$-regular tree and $H$ be a quasi-transitive $d$-regular graph with cycles, then there exists a strong covering map $\pi$ with uniformly non-trivial fibres from $V(T_{d})$ to $V(H).$
 \end{proposition}

\begin{proof}  
We start by constructing a graph $X$ from our graph $H$ which covers $H$ and is isomorphic to $T_{d}$. Fix a vertex $x_0 \in V(H).$ Define the vertices of $X$ to be the non-backtracking paths $<x_0,x_1, \cdots x_n>$ starting at $x_0$ (a path $<x_0,x_1, \cdots x_n>$ is called non-backtracking if $x_{i+2} \neq x_i \ \forall i$). Two paths are connected in $X$ if one is an extension of the other by an edge (this is precisely the universal cover). We claim that for a $d$-regular $H$, $X$ is isomorphic to $T_{d}.$ 

\vspace{2mm}

The fact that $X$ is a tree is clear since all paths are non-backtracking and start at a fix vertex $x_0.$ Now any point $<x_0, x_1, \cdots, x_n>$ has neighbours as $<x_0, x_1 \cdots, x_{n-1}>$ and $<x_0, x_1, \cdots, x_n, u>$ where $u$ runs over all neighbours of $x_n$ not equal to $x_{n-1},$ this shows $d$-regularity. Therefore $X \cong T_{d}.$

\vspace{2mm} For the covering map we let $\pi$ be the map which projects every path to its last vertex, more formally define $\pi: V(T_{d}) \rightarrow V(H) $ such that $\pi(<x_0, x_1, \cdots x_n>) = x_n$ where we identify $T_{d}$ with $X.$ We now show that this is a strong covering map with uniformly non-trivial fibres.

\vspace{2mm}
\textbf{Lipschitz property.} Let $x = \ <x_0, \cdots, x_n>,  y = \ <y_0, \cdots, y_m>$. We want to show that $d_H(\pi(x), \pi(y)) = d_H(x_n,y_m) \leq d_X(x,y)$. Let $z$ be the common ancestor of $x,y$ in $X$. Then $d_X(x,y) = d_X(x,z) + d_X(z,y).$ Since $x$ is a descendant of $z$ it is easy to see that $d_X(x,z) \geq d_H(\pi(x), \pi(z)).$ Thus by the above equation $d_X(x,y) \geq d_H(\pi(x), \pi(y)).$

\vspace{2mm}

\textbf{Uniformly non-trivial fibres.}
We show that there are uniformly non-trivial fibres. This is the only property that requires quasi-transitivity. Pick a $x = <x_0, x_1, \cdots, x_n>$. By quasi-transitivity, we can find a $K$ (independent of $x_n$) such that there is a cycle (not necessarily simple) $C = <x_n, x_{n+1}, \cdots, x_{n+m} = x_n>$ of length $m \leq K.$ If $x_{n-1} = x_{n+1}$, then $y = <x_0, \cdots, x_{n-1}, x_{n+2}, \cdots x_{n+m} = x>$ is a non-backtracking path satisfying $\pi(x) = \pi(y).$ Otherwise, consider the path $y = <x_0, x_1, \cdots, x_{n-1}, x_n, x_{n+1}, \cdots, x_n>$ since $x_{n-1} \neq x_{n+1}$, this is a non-backtracking path and gets mapped $\pi(x)$.

\vspace{2mm}

\textbf{Strong lifting property.}  Pick a $x = <x_0, x_1, \cdots x_n> \ \in 
V(X)$, for any neighbour $u$ of $\pi(x)$ we need to find a neighbour of $x$ mapping to it. If $u = x_{n-1}$ then let that neighbour be $<x_0,x_1, \cdots x_{n-1}>$, otherwise let it be $<x_0,x_1, \cdots x_n,u>$.

\vspace{2mm}
Thus $\pi$ is a strong covering map with uniformly non-trivial fibres, which shows the proposition. By our earlier comments, this also proves Theorem 2.
\end{proof}

\section{Concluding remarks}
Even though we worked with quasi-transitive graphs, the same proof extends to graphs with \textit{bounded local girth}. The concept of bounded local girth can be defined as follows: Consider a vertex $x$, define the \textit{girth} of $x$ as $$L_x:= \inf\{l(C): C \ \text{cycle}\footnote{We are including non-simple cycles, i.e. cycles which visit the same vertex multiple times.}, C \ni x\}$$ where $l(C)$ is the length of the cycle $C$. We say that a graph $G$ has bounded local girth if $\underset{x}{\sup} \ L_x < \infty.$

\vspace{2mm}
The only place where we used quasi-transitivity was to show that our map has uniformly non-trivial fibres. By bounded local girth, the same proof applies, and hence, for any $d$-regular graph $G$, we have $p_c(G) > \frac{1}{d - 1}$. Therefore, trees minimize $p_c$ in the space of all graphs with bounded local girth or no cycles.

\vspace{2mm}
A similar question can be asked for the uniqueness threshold $p_u$. Even though a theorem analogous to Theorem 7 has been established for $p_u$ (see \cite{martineau2019strict}), the same technique cannot be applied since $p_u(T) = 1$ for a tree $T.$

\vspace{2mm}

\section{Acknowledgements}
I thank Prof. Subhajit Goswami for suggesting the problem and for his guidance and comments throughout the project. This work was done as part of the Visiting Students Research Program (VSRP 2024) at the Tata Institute of Fundamental Research, Mumbai and I thank them for this opportunity.

\bibliography{ref}
\bibliographystyle{alpha}
Ishaan Bhadoo, DPMMS, University of Cambridge, UK.\\ \textit{Email address:} \href{ib530@cam.ac.uk}{\textcolor{blue}{ib530@cam.ac.uk}}\\
\end{document}